\documentclass[preprint]{elsarticle}

\usepackage{amssymb}
\usepackage{amsthm}
\usepackage{graphicx}
\usepackage{multirow}
\usepackage{amsmath}
\usepackage{color}

\journal{Statistics \& Probability Letters}

\newtheorem{theorem}{Theorem}
\newtheorem{lemma}[theorem]{Lemma}
\newtheorem{remark}{Remark}

\newcommand{\argmin}{\operatornamewithlimits{arg\ min}}

\begin{document}

\begin{frontmatter}



\title{ First-order weak balanced schemes for bilinear stochastic differential equations \tnoteref{t1}}

\tnotetext[t1]{Partially supported by FONDECYT Grant 1110787. 
Moreover, HAM was partially supported by CONICYT Grant 21090691,
as well as, CMM was partially supported by BASAL Grants PFB-03 and FBO-16.}

\author[label1]{H. A. Mardones } \ead{hmardones@ing-mat.udec.cl}

\author[label1]{C. M. Mora } \ead{cmora@ing-mat.udec.cl}

\address[label1]{Departamento de Ingenier\'{\i}a Matem\'{a}tica and CI$^2$MA, Universidad de Concepci\'{o}n, Chile.}

\begin{abstract}

We use the linear scalar SDE as a test problem
to show that it is possible to construct almost sure stable first-order weak balanced schemes based on the addition of stabilizing functions to the drift terms. Then, we design balanced schemes for multidimensional bilinear SDEs achieving the first order of weak convergence, which do not involve multiple stochastic integrals. To this end, we follow two methodologies to find appropriate stabilizing weights; through an optimization procedure or based on a closed heuristic formula.
Numerical experiments show a promising performance of the new numerical schemes.

\end{abstract}


\end{frontmatter}


\section{Introduction}

Consider the autonomous stochastic differential equation (SDE) 
\begin{equation}
\label{eq:1.1}
 X_{t}=X_{0} + \int_{0}^{t}b\left(  X_{s}\right)  ds+ \sum_{k=1}^{m}\int_{0}^{t}\sigma^{k} \left(X_{s}\right)  dW^{k}_{s},
\end{equation}
where 
$\left(X_t\right)_{t\geq0} $ is an adapted $\mathbb{R}^{d}$-valued stochastic process,
$b, \sigma^k:\mathbb{R}^{d}\rightarrow\mathbb{R}^{d}$ are smooth functions
and
$W^1,\ldots,W^m$ are independent standard Wiener processes.
For solving (\ref{eq:1.1})
in cases the diffusion terms $\sigma^{k}$  play an essential role in the dynamics of $X_t$,
Milstein, Platen and Schurz \cite{Milstein1998} introduced the balanced method
\begin{equation}
 \label{eq:1.7}
\begin{split}
 Z_{n+1}
 & = 
   Z_{n} + b\left(   Z_n \right) \Delta +
  \sum_{k=1}^{m} \sigma^{k} \left( Z_n \right)  \left(  W^k_{\left(n+1 \right) \Delta}-W^k_{n \Delta}\right)
  \\
  & \quad
  + \left(
 c^{0} \left(  Z_n \right) \Delta
  + \sum_{k=1}^{m} c^{k} \left(  Z_n \right) \left\vert W^k_{\left(n+1 \right) \Delta} - W^k_{n \Delta}\right\vert
  \right)
   \left(  Z_n -  Z_{n+1} \right) ,
\end{split}
\end{equation}
where
 $\Delta > 0$ and 
 $c^0,c^1,\ldots,c^m$ are weight functions that should be appropriately chosen for each SDE.
Up to now,
the schemes of type (\ref{eq:1.7})  
use the damping functions $c^1,\ldots,c^m$ to avoid the numerical instabilities caused by $\sigma^k$ or to have $a.s.$ positivity
(see, e.g., \cite{Milstein1998,Alcock2006,Schurz2005,Schurz2012,Tretyakov2013}),
and hence
their rate of  weak convergence is equal to $1/2$, which is low. 
To the best of our knowledge,
concrete balanced versions of the Milstein scheme have been developed only in 
particular cases, like $m=1$,  
where the Milstein scheme does not involve multiple stochastic integrals with respect to different Brownian motions
\cite{Alcock2012,Kahl2006}.

We are interested in the development of efficient first weak order schemes for computing 
$\mathbb{E} f \left( X_t \right) $,
with $f:\mathbb{R}^{d}\rightarrow\mathbb{R}$ smooth.
This motivates the design of balanced schemes based only on  $c_0$,
whose rate of weak convergence  is equal to $1$ under general conditions (see, e.g.,  \cite{Schurz2005,Schurz2012}).
In this direction, \cite{Schurz2005,Schurz2012} propose to take 
 $
 c^0 = 0.5 \, \nabla b
 $,
 together with $c^1 = \cdots = c^m = 0$.
This choice contains no information about the diffusion terms $\sigma^{k}$,
and yields unstable schemes in situations like 
$
 dX_t = \lambda X_{t} \, dW^1_{t} ,
$
with $\lambda > 0$ and $X_t \in \mathbb{R}$;
a test equation used to introduce the balanced schemes (see \cite{Milstein1998}).

This paper addresses the question of whether we can find $c^{0}$ such that 
\begin{equation}
 \label{eq:1.3}
  Z_{n+1} =   Z_{n} 
 		 + b\left(   Z_n \right) \Delta + c^{0} \left(  \Delta,  Z_n \right) \left(  Z_{n+1} - Z_n  \right)  \Delta 
		 + \sum_{k=1}^{m} \sigma^{k} \left( Z_n \right) \sqrt{\Delta}  \xi_{n}^{k} 
\end{equation}
reproduces the long-time behavior of $X_t$,
where from now on $\xi_{0}^{1}, \xi_{0}^{2}, \ldots , \xi_{0}^{m}, \xi_{1}^{1}, \ldots$ are independent random variables 
satisfying $\mathbf{P}\left(  \xi_{n}^{k}= \pm 1\right)  = 1/2$.
Section \ref{sec:LinearScalarSDEs} gives a positive answer to this problem
when (\ref{eq:1.1}) reduces to the classical  scalar SDE
\begin{equation}
 \label{eq:2.1}
X_{t}
=
X_{0}
+ \int_{0}^{t}\mu X_{s}ds
+ \int_{0}^{t}\lambda X_{s}dW^1_{s} 
\end{equation}
where $\mu, \lambda \in \mathbb{R}$.
Indeed,
we obtain an explicit expression for $c^{0} \left(  \Delta, \cdot \right)$
that makes $Z_{n} $  almost sure asymptotically stable for all $\Delta >0$
whenever $2\mu-\lambda^{2} < 0$, as well as positive preserving.
In Section \ref{sec:system},
we propose an optimization procedure for identifying a suitable weight function  $c_0$
in case $b, \sigma^k:\mathbb{R}^{d}\rightarrow\mathbb{R}^{d}$ are linear,
and we also provides a choice of $c_0$ based on a heuristic closed formula.
Both techniques show good results in our numerical experiments,
which encourages further studies of (\ref{eq:1.3}).
All proofs are deferred to Section \ref{sec:Proofs}.

\section{Stabilized Euler scheme for the linear scalar SDE}
\label{sec:LinearScalarSDEs}

In this section $X_t$ satisfies (\ref{eq:2.1}), 
which is a classical test equation for studying
the stability properties of the numerical schemes for (\ref{eq:1.1})
(see, e.g., \cite{Alcock2006,Higham2000,Higham2007}).
We assume,  for simplicity, that $2\mu-\lambda^{2} < 0$.
Set $T_n = n \Delta$, where $ \Delta > 0$ and $n = 0, 1, \ldots $
For all $t\in\left[T_{n},T_{n+1}\right]$ we have 
\[
X_{t}
=
X_{T_n}
+ \int_{T_n}^{t} \left( \mu X_{s} + a \left( \Delta \right) X_s - a \left( \Delta \right) X_s \right) ds
+ \int_{T_n}^{t}\lambda X_{s}dW^1_{s} ,
\]
where $ a \left( \Delta \right)$ is an arbitrary real number.
Then
\[
X_{T_{n+1}}
\approx
X_{T_n}
+  \mu X_{T_n} \Delta + a \left( \Delta \right) \left( X_{T_{n+1}} - X_{T_n} \right) \Delta 
+ \lambda X_{T_n} \left( W^1_{T_{n+1}}  -  W^1_{T_{n}}  \right) ,
\]
and so  $X_t$ is weakly approximated by the recursive scheme
\begin{equation}
\label{eq:3.2}
Y^{s}_{n+1}
=
Y^{s}_{n}
+  \mu Y^{s}_{n} \Delta + a \left( \Delta \right) \left( Y^{s}_{n+1} - Y^{s}_{n} \right) \Delta 
+ \lambda Y^{s}_{n} \sqrt{\Delta} \xi^{1}_{n}.
\end{equation}
In case $a \left( \Delta \right) \Delta \neq 1$, we have
$$
 Y^{s}_{n+1}
=
Y^{s}_{n}
\left(  1+ \left( \mu\Delta+\lambda\sqrt{\Delta} \xi^{1}_{n} \right) / \left(1-a \left( \Delta \right) \Delta \right) \right).
$$
We wish to find a locally bounded function $\Delta \mapsto a \left( \Delta \right) $ such that:
\begin{description}
 \item[P1)]  $Y^{s}_{n}$ preserves $a.s.$ the sign of  $Y^{s}_{0}$ for all $n\in\mathbb{N}$.
 
 \item[P2)] $Y^{s}_{n} $ converges almost surely   to $0$ as $n \rightarrow \infty$ whenever $2\mu-\lambda^{2} < 0$.
\end{description}

We check easily that Property P1 holds iff  
$
 a \left( \Delta \right)
 \in
\left]  -\infty, p_1 \right[  \cup \left]  p_{2} ,+\infty\right[ ,
$
with
$p_{1} := \min \left\{  1 ,   1-\left\vert \lambda\right\vert \sqrt{\Delta}+\mu\Delta  \right\}  / \Delta$
and
$p_{2} :=\max \left\{  1,   1+\left\vert \lambda\right\vert \sqrt{\Delta}+\mu\Delta   \right\}  / \Delta$.
A close look at
$
 \mathbb{E}  \log \left(
 1+ \left( \mu\Delta+\lambda\sqrt{\Delta}\xi^1_{n} \right) / \left( 1-a \left( \Delta \right)  \Delta \right)
  \right)
$
reveals that:

\begin{lemma}
\label{lem:EstLinear}
Suppose that
$a \left( \Delta \right) \Delta\neq1$.
Then,
a necessary and sufficient condition for
Property P1,
together with
$
\lim_{n \rightarrow \infty} Y^{s}_{n}  = 0
$
$a.s.$,
is that
$$
\begin{cases}
 a \left( \Delta \right) \in\left]  -\infty,p_{1} \right[  \cup \left]  p_{2},p_{3} \right[ ,
 &  \text{in case }  \mu<0
 \\
  a \left( \Delta \right) \in\left]  -\infty,p_{1} \right[  \cup\left] p_{2},+\infty\right[ ,
 & \text{in case }  \mu=0 \text{ and } \lambda \neq 0

  \\
a \left( \Delta \right) \in\left]  p_{3} ,p_{1} \right[ \cup\left]  p_{2}, +\infty\right[  ,
 & \text{in case } \mu > 0
\end{cases}
$$
where
$p_{3} := \left( \mu^{2}\Delta+2\mu-\lambda^{2}\right)  /\left(  2\mu\Delta\right)  $.
\end{lemma}

Using Lemma  \ref{lem:EstLinear} we deduce that we can choose
 \begin{equation}
\label{eq:2.7}
 a \left( \Delta \right)
=
\begin{cases}
 \mu - \alpha_1  \left( \Delta \right) \lambda^{2}
,
 & \text{if } \mu \leq 0
 \\
   \mu -  \alpha_2  \left( \Delta \right) \lambda^{2}
,
 & \text{if }  \mu > 0  \text{ and }   \Delta < 2/ \mu
  \\
  \left( 1+\left\vert \lambda \right\vert \sqrt{\Delta}+ \mu \Delta  \right)  / \Delta + \beta
   ,
 & \text{if }  \mu > 0 \text{ and } \Delta  \geq  2/ \mu
\end{cases}
\end{equation}
where $\beta >0$,
$
1/4
 <  \alpha_2  \left( \Delta \right)
\leq 1/4 +  \left(  \lambda^2  - 2 \mu \right) \left(  2  -  \mu \Delta \right)  / \left( 8 \lambda^2 \right)
$
and
$\alpha_1$ is a bounded function satisfying 
$\alpha_1  \left( \Delta \right) > 1/4$.

\begin{theorem}
\label{th:EstLinear}
Let $2\mu-\lambda^{2} < 0$.
Then,
$Y^{s}_{n}$ with $a \left( \Delta \right)$ given by (\ref{eq:2.7})
satisfies Properties P1 and P2.
\end{theorem}

\begin{figure}[tb]
  \begin{center}
      \includegraphics[height= 3.2in,width=5.2in]{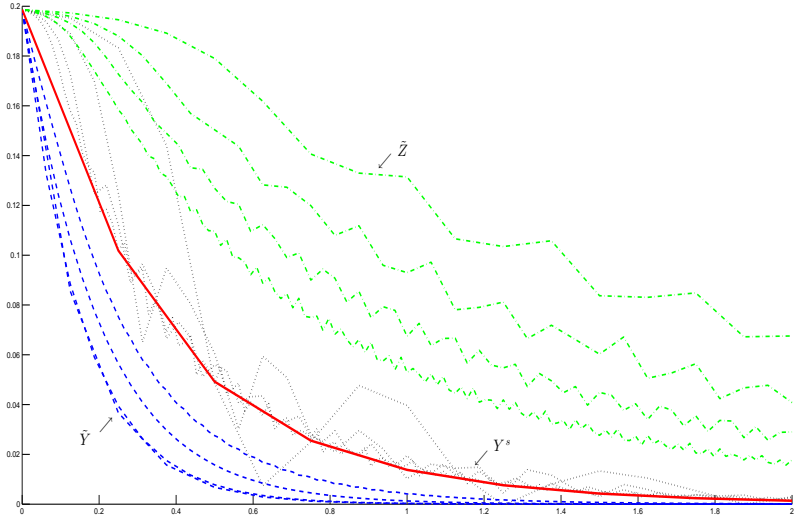}
    \caption{ Computation of  $ \mathbb{E} \sin\left( X_t  / 5 \right) $,
where $t\in\left[0, 2\right]$ and $X_t$ solves (\ref{eq:2.1}) with
$\mu = 0$, $\lambda  = 4$ and $X_0 = 1$. 
Dashed  line: $\tilde{Y}$, dashdot line: $\tilde{Z}$, dotted line: $Y^{s}$, and solid line: reference values.
Here, $\Delta$ takes the values $1/8$, $1/16$, $1/32$ and $1/64$.
As we expected, smaller $\Delta$ produce better approximations.}
    \label{fig:1}
  \end{center}
\end{figure}

Following  \cite{Milstein1998},
we now illustrate the behavior of $Y^{s}_{n}$ using  (\ref{eq:2.1}) with $\mu = 0$ and $\lambda  = 4$.
We take $X_0 = 1$.
Since $\mu \leq 0 $, 
we choose  $\alpha_1  \left( \Delta \right) = 1/4 + 1/100$;
its convenient to keep  the weights as small as possible.
Figure \ref{fig:1} displays the computation of
$\mathbb{E} \sin\left( X_t  / 5 \right)$ obtained from the sample means of $25 \cdot 10^9$ observations of:
 $Y^{s}_{n}$ with $a \left( \Delta \right) = - 0.26 \ \lambda^2$, 
 the fully implicit method 
$
 \tilde{Y}_{n+1}
 =
 \tilde{Y}_{n} /
 \left(
 1 + \lambda^2 \Delta - \lambda \sqrt{\Delta}  \xi^1_{n}
 \right)
$
(see p.  497 of  \cite{Kloeden1992}),
and the balanced scheme 
\[
\widetilde{Z}_{n+1}
=
\widetilde{Z}_{n}\left(  1+\lambda\sqrt{\Delta} \xi^{1}_{n} +\lambda\sqrt{\Delta}\right)
/
\left(  1 + \lambda\sqrt{\Delta}\right),
\]
which is a weak version of the method developed in Section 2 of  \cite{Milstein1998}.
$\widetilde{Z}_{n}$ preserves  the sign of $\widetilde{Z}_{0}$ and is almost sure asymptotically stable (see
\cite{Schurz2012}).
Solid line identifies the `true' values gotten by sampling $25 \cdot 10^9$ times
$
\exp\left( - 8 t   + 4 W_{t}\right)
$.

In contrast with the poor performance of the Euler-Maruyama scheme
when
the step sizes are greater than or equal to $1/16$,
Figure \ref{fig:1}
suggests us that
$Y^{s}_{n}$ is an efficient scheme having good qualitative and convergence properties.
In this numerical experiment,
the accuracy of $\tilde{Z}_{n}$ is not good,
and 
$\tilde{Y}_{n}$ decays  too fast to $0$ as $n \rightarrow \infty$.

\section{System of bilinear SDEs}
\label{sec:system}

This section is devoted to the  SDE
\begin{equation}
 \label{eq:4.1}
 X_{t}=X_{0}+\int_{0}^{t}BX_{s}ds+\sum_{k=1}^{m}\int_{0}^{t}\sigma^{k} X_{s}dW_{s}^{k},
\end{equation}
where $X_t \in \mathbb{R}^d$ and  $B,\sigma^{k} \in  \mathbb{R}^{d \times d}$.
The bilinear SDEs describe dynamical features of non-linear SDEs via the linearization around their equilibrium points (see, e.g., \cite{Baxendale1994}).
The system of SDEs (\ref{eq:4.1}) also appears, for example, 
in the spatial discretization of stochastic partial differential equations (see, e.g., \cite{Gyongy2002,Jentzen2011}).

\subsection{Heuristic balanced scheme}

Since (\ref{eq:4.1}) is bilinear,
we restrict $c^{0}$ to be constant,
and so  (\ref{eq:1.3}) becomes
\begin{equation}
\label{eq:4.2}
 Z_{n+1} 
= 
Z_{n} + B Z_{n} \Delta + H  \left( \Delta \right) \left( Z_{n+1} - Z_{n} \right)  \Delta 
+ 
\sum_{k=1}^{m} \sigma^{k} Z_{n} \sqrt{\Delta}  \xi^k_{n} ,
\end{equation}
with $H : \left] 0 , \infty \right[ \rightarrow \mathbb{R}^{d\times d}$ and $\Delta > 0$.
The rate of weak convergence of $Z_n$ is equal to $1$ provided, for instance, that
$ H \left( \Delta \right)$ and $\left( I - \Delta H \left( \Delta \right) \right)^{-1}$ are bounded on any  interval 
$\Delta \in \left] 0, a \right]$
(see, e.g., \cite{Schurz2005}).
Generalizing roughly Section \ref{sec:LinearScalarSDEs} we choose 
$
H  \left( \Delta \right) =B-  \sum\limits_{k=1}^{m} \alpha_{k} \left( \Delta \right) \left(  \sigma^{k}\right)  ^{\top} \sigma^{k},
$
where, for example, $\alpha_{k} \left( \Delta \right) =0.26$.
This gives the recursive scheme
\begin{equation}
\label{eq:4.3}
\begin{split}
\left( I - \Delta B +  0.26 \ \Delta \sum\limits_{k=1}^{m}  \left(  \sigma^{k}\right)  ^{\top} \sigma^{k} \right)
Y^{s}_{n+1} 
& = 
Y^{s}_{n} 
+ 0.26 \ \Delta \sum\limits_{k=1}^{m}  \left(  \sigma^{k}\right)  ^{\top} \sigma^{k}  Y^s_n
\\
& \quad
+ 
\sum_{k=1}^{m} \sigma^{k} Y^{s}_{n} \sqrt{\Delta}  \xi^k_{n} ,
\end{split}
\end{equation}
which is a first-order weak balanced version of the semi-implicit Euler method. 

\begin{remark}
Combining (\ref{eq:4.3}) with ideas of the local linearization method
(see, e.g., \cite{Biscay1996,DeLaCruz2010})
we deduce the following numerical method for (\ref{eq:1.1}):
\begin{equation*}
\label{eq:heuristic_nonlinear}
U_{n+1}=U_n+b\left(U_n\right)\Delta+\sum_{k=1}^{m}\sigma^k\left(U_n\right)\sqrt{\Delta}  \xi^k_{n}+ H  \left( \Delta, U_n \right) \left( U_{n+1} - U_{n} \right)
\Delta,
\end{equation*}
where
$
H  \left( \Delta,x\right) =\nabla b\left(x\right)-
\sum\limits_{k=1}^{m} \alpha_{k} \left( \Delta \right) \left(\nabla
\sigma^{k}\left(x\right)\right)  ^{\top}\nabla \sigma^{k}\left(x\right)
$
with $\alpha_{k} \left( \Delta \right) =0.26$.
\end{remark}

\subsection{Optimal criterion  to select $c_0$}

In case $I  -  \Delta H  \left( \Delta \right)$ is invertible,
according to (\ref{eq:4.2}) we have
\begin{equation}
 \label{eq:4.5}
 Z_{n+1} 
= 
Z_n + \left( I  -  \Delta H \left( \Delta \right) \right)^{-1} \left(   \Delta B   + \sum_{k=1}^{m}  \sqrt{\Delta}  \xi^k_{n} \sigma^{k} \right) Z_{n},
\end{equation}
where $I$ is the identity matrix.
Therefore,
a more general formulation of $Z_n$ is 
\begin{equation}
 \label{eq:4.6}
V_{n+1} 
= 
V_n + \left( I  +  \Delta M \left( \Delta \right) \right) \left(  \Delta B   + \sum_{k=1}^{m}  \sqrt{\Delta}  \xi^k_{n} \sigma^{k} \right) V_n,
\end{equation}
with  $M : \left] 0 , \infty \right[ \rightarrow \mathbb{R}^{d\times d}$.
In fact, taking 
$
 M \left( \Delta \right)
 =
 \left( \left( I  -  \Delta H \left( \Delta \right) \right)^{-1} - I \right)/\Delta
$
we obtain (\ref{eq:4.5}) from (\ref{eq:4.6}).
The following theorem provides a useful estimate of 
the growth rate of $V_{n}$
in terms of $\mathbb{E} \log \left( \left\Vert A_0  \left( \Delta, M  \left( \Delta \right) \right) x \right\Vert \right)$,
a quantity that we can compute explicitly  in each specific situation.

\begin{theorem}
\label{th:EstSistema}
Let $V_n$ be defined recursively by (\ref{eq:4.6}).
Then 
\begin{equation}
\label{eq:3.3}
\lim_{n\rightarrow \infty}\frac{1}{n\Delta}\log \left( \left\Vert V_n \right\Vert \right)
\leq 
\frac{1}{\Delta} \sup_{x\in\mathbb{R}^d,\left\Vert x \right\Vert=1}\mathbb{E}
\log \left( \left\Vert A_n \left(  \Delta, M  \left( \Delta \right) \right) x \right\Vert \right),
\end{equation}
where
$A_n  \left(  \Delta, M  \right) = 
I + \left( I  +  \Delta M \right) \left(  \Delta B   + \sum_{k=1}^{m}  \sqrt{\Delta}  \xi^k_{n} \sigma^{k} \right) $.
\end{theorem}

Set 
$
\ell :=
 \sup_{x\in\mathbb{R}^d, \left\| x \right\| = 1} 
 \left(
 \langle x , B x \rangle
 +
\frac{1}{2}\sum_{k=1}^m \left\Vert \sigma^k x \right\Vert^2
-
 \sum_{k=1}^m  \langle x ,\sigma^k x \rangle^2
\right)
$.
Then 
\begin{equation}
\label{eq:4.4}
\limsup_{t\rightarrow \infty}\frac{1}{t}\log\left(\left\Vert X_t\right\Vert\right) \leq \ell 
\hspace{1cm}  a.s.
\end{equation}
(see, e.g., \cite{Higham2007}).
Fix $\Delta > 0$.
We would like that for all $ \left\| x \right\| = 1$,
$$
\frac{1}{\Delta} \mathbb{E} \log \left( \left\Vert A_0  \left(  \Delta, M  \left( \Delta \right) \right) x \right\Vert \right)
\approx
\langle x , B x \rangle
 +
\frac{1}{2}\sum_{k=1}^m \left\Vert \sigma^k x \right\Vert^2
-
 \sum_{k=1}^m  \langle x ,\sigma^k x \rangle^2  .
$$
A simpler problem is to find $M \left( \Delta \right)$ for which 
the upper bounds  (\ref{eq:3.3}) and (\ref{eq:4.4}) are as close as possible,
and so we can expect that 
$V_n$ inherits the long-time behavior of $X_t$.
Then, 
we propose to take 
\begin{equation}
 \label{eq:4.7}
 M  \left( \Delta \right) 
\in
\argmin 
\left\{
\left(
 \frac{1}{\Delta} \sup_{x\in\mathbb{R}^d,\left\Vert x \right\Vert=1}
 \mathbb{E} \log \left( \left\Vert A_0  \left(  \Delta, M \right) x \right\Vert \right)
 -
\ell
\right)^2
:
M \in  \mathcal{M}
\right\} ,
\end{equation}
where  $\mathcal{M}$ is a predefined subset $\mathbb{R}^{d\times d}$.
Two examples of $\mathcal{M}$ used successfully  in our numerical experiments are 
$\mathbb{R}^{d\times d}$
and 
$
\left\{ \left( M_{i,j} \right)_{1 \leq i,j \leq d }
: 
\left| M_{i,j} \right| \leq K  \text{ for all } i,j  \right\} 
$,
with $K$ large enough.
Applying the classical methodology introduced by Talay and Milshtein
for studying the weak convergence order 
(see, e.g., \cite{Graham2013,Milstein2004})
we can deduce that $V_n$ converges weakly with order  $1$ whenever 
$\Delta\rightarrow M\left(\Delta\right)$ is locally bounded.

\begin{remark}
Sometimes, the asymptotic behavior of (\ref{eq:1.1}) depends on the properties of the
SDE  obtained by linearizing (\ref{eq:1.1}) around $0$
(see, e.g., \cite{Baxendale1994}).
In these cases, we can extend to (\ref{eq:1.1})  the scheme given by (\ref{eq:4.6}) and (\ref{eq:4.7}) as 
 \begin{equation*}
V_{n+1}
=
V_n + \left( I  +  \Delta M \left( \Delta \right) \right) \left(
b\left( V_{n}\right) \Delta+
\sum_{k=1}^{m} \sigma^{k} \left(V_{n}\right) \sqrt{\Delta}  \xi^k_{n} \right)
\end{equation*}
where now $M \left( \Delta \right) $ is described by (\ref{eq:4.7}) with
$\mathcal{M}$ a predefined subset of $\mathbb{R}^{d\times d}$,
$$
\ell :=
\sup_{x\in\mathbb{R}^d, \left\| x \right\| = 1}
\left(
\langle x ,  \nabla b\left(0\right) x \rangle
+
\frac{1}{2}\sum_{k=1}^m \left\Vert \nabla \sigma^k\left(0\right) x
\right\Vert^2
-
\sum_{k=1}^m  \langle x ,\nabla \sigma^k\left(0\right) x \rangle^2
\right)
$$
and
$
A_0  \left(  \Delta, M  \right) =
I + \left( I  +  \Delta M \right) \left(   \nabla b\left(0\right)
\Delta  + \sum_{k=1}^{m}  \nabla \sigma^{k}\left(0\right)\sqrt{\Delta}
\xi^k_{0} \right)
$.
\end{remark}

\subsection{Numerical experiment}

\renewcommand{\arraystretch}{1.3}
\begin{table}[bt]
\begin{center} \footnotesize
\begin{tabular}{|c|cccccc|}
\hline
$\Delta$ & $1/2$ & $1/4$ & $1/8$ & $1/16$ & $1/32$ & $1/64$\\ 
\hline
$ M_{1,1} \left( \Delta \right)$ & $-1.6099$ & $-5.1036$ & $-4.8804$ & $-7.1499$ & $-1.6758$ & $0.9887$\\ 
\hline
$ M_{2,1} \left( \Delta \right)$ & $0.0975$ & $0.2758$ & $0.7667$ & $1.0136$ & $1.1500$ & $0.9918$\\ 
\hline
$ M_{1,2} \left( \Delta \right)$ & $-0.0975$ & $-0.2752$ & $-0.8505$ & $-0.1814$ & $-1.0448$ & $-1.9947$\\ 
\hline
$ M_{2,2} \left( \Delta \right)$ & $-1.3173$ & $-5.9305$ & $-2.6136$ & $-2.3003$ & $-1.7421$ & $-1.9005$\\ 
\hline
Order & $ -10$ & $-19$ & $-21$ & $-21$ & $-20$ & $-19$\\ 
\hline
\end{tabular}
\caption{Approximate values of the weight matrix $\left( M_{i,j} \left(\Delta\right) \right)_{1 \leq i,j \leq 2 }$ for (\ref{EDE}) with $\sigma_1 = 7$, $\sigma_2 = 4$ and $\epsilon = 1$,
together with the corresponding order of magnitude of the objective function minimum.
}

\label{Tabla2}
\end{center}
\end{table}

We consider the non-commutative test equation
\begin{equation}
\label{EDE}
dX_{t} 
=
\begin{pmatrix}
 \sigma_1  & 0 \\
0 & \sigma_2 %
\end{pmatrix}
 X_{t} \, dW_{t}^{1} 
+
\begin{pmatrix}
0 & -\epsilon  \\
\epsilon  & 0%
\end{pmatrix}%
X_{t} \, dW_{t}^{2},  
\end{equation}%
where  $\sigma_1=7$, $\sigma_2=4$, $\epsilon=1$ and $X_0=\left(1,2\right)^{\top}$.
Since $0 < \sigma_2 < \sigma_1 < 3 \sigma_2$,
applying elementary calculus we get
 $\ell =\left(\epsilon ^2 - \sigma_2 ^2\right)/2 < 0$,
and so $X_t$ converges exponentially fast to $0$.
To illustrate the performance of schemes of type (\ref{eq:1.3}),
we take $V_n$ defined by (\ref{eq:4.6}) and (\ref{eq:4.7})
with  
$
\mathcal{M} =  
\left\{ \left( M_{i,j} \right)_{1 \leq i,j \leq 2 }
: 
\left| M_{i,j} \right| \leq 20  \text{ for all } i,j  \right\} 
$.
Table \ref{Tabla2} provides four-decimal approximations of the components of $M  \left( \Delta \right)$,
which have been obtained by running ($5^4$-times) the MATLAB function \texttt{fmincon} for the initial
parameters 
$
\left\{ \left( M_{i,j} \right)_{1 \leq i,j \leq 2 }
: 
M_{i,j} \in \left\{ -2, -1, 0, 1 , 2 \right\}  \text{ for all } i,j \right\} 
$.

 \begin{figure}[tb]
  \begin{center}
       \includegraphics[height= 3.2in,width=5.2in]{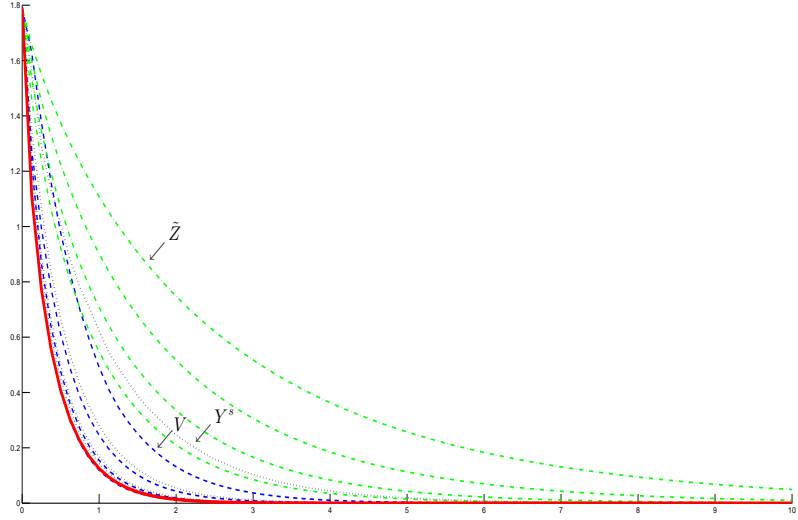}
    \caption{ Computation of  $\mathbb{E} \log\left(1+ \left\Vert X_t\right\Vert^2  \right) $,
where $t\in\left[0, 10\right]$ and $X_t$ solves (\ref{EDE}).
Dashed  line: $V_n$, dashdot line: $\tilde{Z}$, dotted line: $Y^{s}$, and solid line: reference values.
Here, $\Delta$ is equal to $1/8$, $1/16$, $1/32$ and $1/64$;
smaller discretization steps produce better approximations.
}
    \label{fig:2}
  \end{center}
\end{figure}

 Figure \ref{fig:2} shows the computation $\mathbb{E} \log\left(1+ \left\Vert X_t\right\Vert^2  \right)$  
 by means of $V_n$ (dashed line), $Y^s_n$ (dotted line), and the weak  balanced scheme (dashdot line)
\begin{align*}
\widetilde{Z}_{n+1}
& =  \widetilde{Z}_n+
\begin{pmatrix}
\sigma_1  & 0 \\
0 & \sigma_2 %
\end{pmatrix}%
\widetilde{Z}_n\sqrt{\Delta} \xi_n^1+
\begin{pmatrix}
0 & -\epsilon  \\
\epsilon  & 0%
\end{pmatrix}%
\widetilde{Z}_n\sqrt{\Delta}\xi_n^2
\\
& \quad
+
\sqrt{\Delta} 
\begin{pmatrix}
\left| \sigma_1 \right| + \left| \epsilon \right| & 0 \\
0 & \left| \sigma_2 \right| + \left| \epsilon \right|%
\end{pmatrix}%
\left( \widetilde{Z}_n -  \widetilde{Z}_{n+1} \right) 
\end{align*}
(see \cite{Schurz2005}).
The reference values for $\mathbb{E} \log\left(1+ \left\Vert X_t\right\Vert^2  \right)$  (solid line) 
have been calculated by using the weak Euler method 
 $$
\widetilde{Y}_{n+1}=\widetilde{Y}_n+
\begin{pmatrix}
\sigma_1  & 0 \\
0 & \sigma_2 %
\end{pmatrix}%
\widetilde{Y}_n\sqrt{\Delta} \, \xi_n^1+
\begin{pmatrix}
0 & -\epsilon  \\
\epsilon  & 0%
\end{pmatrix}%
\widetilde{Y}_n\sqrt{\Delta} \, \xi_n^2
$$
with step-size $\Delta=2^{-13}\approx0.000122$.
Indeed,
we plot the sample means obtained from $10^{8}$ trajectories of each scheme.
Furthermore, Table \ref{TablaBilineal} provides estimates of the errors  
$
\epsilon \left( \hat{Y} \right)
: =
\left| 
\mathbb{E} \log\left(1+ \left\Vert X_T \right\Vert^2  \right) 
-
\mathbb{E} \log\left(1+ \left\Vert \hat{Y}_{N} \right\Vert^2  \right)
\right| 
$,
where $T=1,3$, $N = T / \Delta$,
and $\hat{Y}_n$ represents  the numerical methods
$V_n$, $Y^s_n$, $\widetilde{Y}_n$ and $\widetilde{Z}_n$.
From Table \ref{TablaBilineal} we can see that  $\widetilde{Y}_n$ blows up for $\Delta \leq 1/16$. 
Figure  \ref{fig:2}, together with Table \ref{TablaBilineal}, illustrate that $\widetilde{Z}_{n}$ is stable, 
but presents a slow rate of weak convergence.
In contrast,  
the performance of $V_n$ is very good,
$V_n$ mix good stability properties with  reliable approximations. 
The heuristic balanced scheme $Y^s_n$ shows a very good behavior.
In fact, the accuracy of $Y^s_n$  is very similar to that of $V_n$ for $\Delta \leq 1/16$,
and  $Y^s_n$ does not involve any optimization process.

\renewcommand{\arraystretch}{1.3}
\begin{table}[bt]

\begin{center} \footnotesize

\begin{tabular}{cc|c|c|c|c|c|c|}
\cline{3-8}
& & \multicolumn{6}{ c| }{$ \Delta $} \\
&   & 1/2 & 1/4 & 1/8 & 1/16& 1/32 & 1/64  \\ 
\cline{1-8}
\multicolumn{1}{ |c }{\multirow{2}{*}{ $ \epsilon \left( \tilde{Y} \right) $} } &
\multicolumn{1}{ c| }{ $ T = 1$ }   & 6.5497  & 9.4879  & 12.733 & 11.0676  & 0.15183  & 0.02365       \\ 
\multicolumn{1}{ |c  }{}                        &
\multicolumn{1}{ c| }{ $ T = 3$ }   & 18.814  & 28.8744  & 38.9743 & 34.1327  & 0.0086188  & 0.00075718   \\  \hline 
\multicolumn{1}{ |c  }{\multirow{2}{*}{ $ \epsilon \left( \tilde{Z} \right) $ } } &
\multicolumn{1}{ c| }{ $ T = 1$ }  & 1.3395  & 1.1777  & 0.98272 & 0.7757  & 0.58279  & 0.42137   \\ 
\multicolumn{1}{ |c  }{}                        &
\multicolumn{1}{ c| }{ $ T = 3$ }  & 1.0611  & 0.78255  & 0.51624 & 0.30475  & 0.1643  & 0.08361   \\  \hline 
\multicolumn{1}{ |c  }{\multirow{2}{*}{ $ \epsilon \left( Y^s \right) $ } } &
\multicolumn{1}{ c| }{ $ T = 1$ }  & 1.1914  & 0.85936  & 0.49789 & 0.15466  & 0.042484  & 0.018271  \\ 
\multicolumn{1}{ |c  }{}                        &
\multicolumn{1}{ c| }{ $ T = 3$ } & 0.81853  & 0.38585  & 0.10185 & 0.0096884  & 0.0013717  & 0.00055511    \\  \hline 
\multicolumn{1}{ |c  }{\multirow{2}{*}{ $ \epsilon \left( V \right) $ } } &
\multicolumn{1}{ c| }{ $ T = 1$ }  & 1.2544  & 0.8482  & 0.36579 & 0.11998  & 0.029324  & 0.0069274  \\ 
\multicolumn{1}{ |c  }{}                        &
\multicolumn{1}{ c| }{ $ T = 3$ }  & 0.64867  & 0.16695  & 0.035366 & 0.0065051  & 0.00068084  & 0.00031002  \\  \hline
\end{tabular}
\end{center}

\caption{Estimation of errors involved in the computation of 
 $\mathbb{E} \log\left(1+ \left\Vert X_T\right\Vert^2  \right)$  for $T=1$ and $T=3$.
Here, $X_t$ verifies (\ref{EDE}) with $\sigma_1=7$, $\sigma_2=4$, $\epsilon =1$ and $X_0 =\left(1,2\right)^{\top}$.
\label{TablaBilineal}
}
\end{table}

\section{Proofs}
\label{sec:Proofs}

\begin{proof}[Proof of Lemma \ref{lem:EstLinear}]

We first prove that 
under Property P1,
 $
\lim_{n \rightarrow \infty} Y^{s}_{n}  = 0
$
$a.s.$
iff
\begin{equation}
 \label{eq:2.2}
\begin{cases}
 a \left( \Delta \right)
<
\left( \mu^{2}\Delta+2\mu-\lambda^{2} \right) / \left(  2\mu\Delta \right)
,
 & \text{if } \mu<0
 \\
   a \left( \Delta \right) \in\mathbb{R},
 & \text{if } \mu=0 \text{ and } \lambda \neq 0
  \\
 a \left( \Delta \right)
 >
 \left( \mu^{2}\Delta+2\mu-\lambda^{2} \right) / \left(  2\mu\Delta \right) ,
 & \text{if } \mu > 0
\end{cases}.
\end{equation}
Suppose that Property P1 holds.
Applying the strong law of large numbers and the law of iterated logarithm
we obtain that
$ Y^{s}_{n}  \rightarrow 0$  $a.s.$ as $n \rightarrow \infty$
iff
\begin{equation}
 \label{eq:2.3}
 \mathbb{E}  \log \left(
 1+ \left( \mu\Delta+\lambda\sqrt{\Delta}\xi^1_{n} \right) / \left( 1-a \left( \Delta \right)  \Delta \right)
  \right)
 < 0
\end{equation}
(see, e.g., Lemma 5.1 of \cite{Higham2000}).
 Since
\[
 \mathbb{E}  \log\left(  1+\frac{\mu\Delta+\lambda\sqrt{\Delta}\xi^1_{n}}{1-a \left( \Delta \right)  \Delta}  \right)
 =
 \frac{1}{2}\log\left(  \left(  1+\frac{\mu\Delta}{1-a \left( \Delta \right)  \Delta
}\right)  ^{2}-\frac{\lambda^{2}\Delta}{  \left( 1-a \left( \Delta \right) \Delta \right)
^{2}}  \right) ,
\]
inequality (\ref{eq:2.3}) becomes
$
2 \mu\left(  1-a \left( \Delta \right) \Delta\right)
+ \mu^{2} \Delta -  \lambda^{2}
<
0
$,
which is equivalent  to (\ref{eq:2.2}).
This establishes our first claim.

From the assertion of the first paragraph 
we get that
Property P1,
together with
$
\lim_{n \rightarrow \infty} Y^{s}_{n}  = 0
$
$a.s.$,
is equivalent to
(a) $a \left( \Delta \right) \in\left]  -\infty, \min\{p_{1}, p_{3} \} \right[  \cup\left]  p_{2},p_{3}\right[  $ for $\mu<0$;
(b) $a \left( \Delta \right) \in\left]  -\infty,p_{1} \right[  \cup\left] p_{2},+\infty\right[  $  for $\mu=0$ and $\lambda \neq 0$;
and
$a \left( \Delta \right) \in\left]  p_{3},p_{1} \right[ \cup\left]  \max \{ p_{2}, p_{3} \} ,+\infty\right[  $ for $\mu>0$.
This gives the lemma,
because $p_1 < p_3$ (resp. $p_2 > p_3$)
whenever  $\mu < 0$ (resp.  $\mu > 0$).
\end{proof}

\begin{proof}[Proof of Theorem \ref{th:EstLinear}]
In case $\lambda \neq 0$,
using differential calculus we obtain that the function
$
\Delta
\mapsto
\left(  1-\left\vert \lambda\right\vert \sqrt{\Delta}+\mu\Delta\right)  /\Delta
$
attains its global minimum at $4/ \lambda^2$.
Then,  for all $\Delta > 0$ and $\lambda \in \mathbb{R}$ we have
\begin{equation}
\label{eq:2.5}
 \left(  1-\left\vert \lambda\right\vert \sqrt{\Delta}+\mu\Delta\right)  /\Delta
\geq
\mu - \lambda^2 /4 .
\end{equation}
First,
we suppose that $\mu \leq 0$ and $\alpha_1  \left( \Delta \right) > 1/4$.
From (\ref{eq:2.5})
it follows that
$p_1 >  \mu - \alpha_1  \left( \Delta \right) \lambda^{2}$,
which implies
$
a \left( \Delta \right)
\in
\left]  -\infty,p_{1} \right[
$.
Second,
if
$\mu > 0$
and  $\Delta  \geq  2/ \mu $,
then
$
a \left( \Delta \right)
\in
\left]  p_{2} , + \infty  \right[
$.
Third, assume that
$\mu > 0$
and
$\Delta < 2/ \mu $.
Since $\mu > 0$,
for any
$ \Delta < \lambda^2 / \mu^2 $
we have
$
 1-\left\vert \lambda\right\vert \sqrt{\Delta}+\mu\Delta < 1
$.
Using
$2\mu-\lambda^{2}<0$
we get
$\lambda^2 / \mu^2 > 2 / \mu$,
and so
$
p_1
=
\left( 1-\left\vert \lambda\right\vert \sqrt{\Delta}+\mu\Delta \right) / \Delta
$
whenever
$\Delta < 2/ \mu $.
Applying (\ref{eq:2.5}) gives
$
p_1
>
\mu - \alpha_2  \left( \Delta \right) \lambda^2
$,
because $\alpha_2  \left( \Delta \right) > 1/4$.
On the other hand,
we have
$
p_3 < \mu - \alpha_2  \left( \Delta \right) \lambda^2
$
if and only if
$
 2 \mu  - \lambda^2
 <
 \mu \Delta \left( 2\mu - 4 \alpha_2   \left( \Delta \right) \lambda^2  \right)/2
 $,
 which becomes
\begin{equation}
 \label{eq:2.8}
\frac{2}{\mu}
>
 \Delta
 \left(
 1
 + \left(4 \alpha_2  \left( \Delta \right) -1\right) \frac{ \lambda^2} {  \lambda^2 - 2 \mu }
 \right)
\end{equation}
since $2\mu-\lambda^{2} < 0$ and $\mu > 0$.
By $2 / \mu >  \Delta$,
 (\ref{eq:2.8})  holds in case
$
 \alpha_2  \left( \Delta \right)
\leq 1/4 +  \left(  \lambda^2  - 2 \mu \right) \left(  2  -  \mu \Delta \right)  / \left( 8 \lambda^2 \right)
$.
Then
$
p_3 < \mu - \alpha_2 \left( \Delta \right)  \lambda^2
$,
hence
$
a \left( \Delta \right)
\in
\left]  p_{3} , p_{1} \right[
$.
Combining Lemma  \ref{lem:EstLinear} with the above three cases
yields Properties P1 and P2.
\end{proof}

\begin{proof}[Proof  of Theorem \ref{th:EstSistema}]

From (\ref{eq:4.6}) it follows that
$$
V_n = A_{n-1} \left(  \Delta, M  \left( \Delta \right) \right) A_{n-2}  \left(  \Delta, M  \left( \Delta \right) \right)  \cdots A_0  \left(  \Delta, M  \left( \Delta \right) \right) V_0 .
$$
Since $\xi^k_{n}$ are bounded,
$
 \sup_{x\in\mathbb{R}^d,\left\Vert x \right\Vert=1} 
 \mathbb{E} \log_{+} \left( \left\Vert A_0  \left( \Delta,  M  \left( \Delta \right) \right) x \right\Vert \right) 
< \infty,
$
where $\log_{+} \left( x \right) $ stands for the positive part of $\log \left( x \right) $.
Hence, 
$\lim_{n\rightarrow \infty}\frac{1}{n}\log \left( \left\Vert V_n \right\Vert \right)$ exists 
whenever $V_0 \neq 0$,
and only dependents on $V_0$.
Furthermore,
$$
\lim_{n\rightarrow \infty}\frac{1}{n}\log \left( \left\Vert V_n \right\Vert \right)
=
\int_{\left\Vert x \right\Vert=1}
\mathbb{E} \log \left( \left\Vert A_0 \left(  \Delta, M  \left( \Delta \right) \right) x \right\Vert \right) \mu \left( dx \right),
$$
with $\mu$  probability measure (see, e.g.,  Theorem $3.1$ of \cite{Cohen1984}).
This gives (\ref{eq:3.3}).
\end{proof}




\begin{thebibliography}{10}
\expandafter\ifx\csname url\endcsname\relax
  \def\url#1{\texttt{#1}}\fi
\expandafter\ifx\csname urlprefix\endcsname\relax\def\urlprefix{URL }\fi
\expandafter\ifx\csname href\endcsname\relax
  \def\href#1#2{#2} \def\path#1{#1}\fi

\bibitem{Milstein1998}
G.~N. Milstein, E.~Platen, H.~Schurz, Balanced implicit methods for stiff
  stochastic systems, SIAM J. Numer. Anal. 35 (1998) 1010--1019.

\bibitem{Alcock2006}
J.~Alcock, K.~Burrage, A note on the {B}alanced method, BIT 46 (2006) 689--710.

\bibitem{Schurz2005}
H.~Schurz, Convergence and stability of balanced implicit methods for systems
  of {SDE}s, Int. J. Numer. Anal. Model. 2 (2005) 197--220.

\bibitem{Schurz2012}
H.~Schurz, Basic concepts of numerical analysis of stochastic differential
  equations explained by balanced implicit theta methods, in: M.~Zili, D.~V.
  Filatova (Eds.), Stochastic Differential Equations and Processes, Springer,
  New York, 2012, pp. 1--139.

\bibitem{Tretyakov2013}
M.~V. Tretyakov, Z.~Zhang, A fundamental mean-square convergence theorem for
  {SDE}s with locally {L}ipschitz coefficients and its applications, SIAM J.
  Numer. Anal. 51 (2013) 3135--3162.

\bibitem{Alcock2012}
J.~Alcock, K.~Burrage, Stable strong order 1.0 schemes for solving stochastic
  ordinary differential equations, BIT 52 (2012) 539--557.

\bibitem{Kahl2006}
C.~Kahl, H.~Schurz, Balanced {M}ilstein methods for ordinary {SDEs}, Monte
  Carlo Methods Appl. 12 (2006) 143--170.

\bibitem{Higham2000}
D.~J. Higham, Mean-square and asymptotic stability of the stochastic theta
  method, SIAM J. Numer. Anal. 38 (2000) 753--769.

\bibitem{Higham2007}
D.~J. Higham, X.~Mao, C.~Yuan, Almost sure and moment exponential stability in
  the numerical simulation of stochastic differential equations, SIAM J. Numer.
  Anal. 45 (2007) 592--609.

\bibitem{Kloeden1992}
P.~E. Kloeden, E.~Platen, Numerical solution of stochastic differential
  equations, Springer-Verlag, Berlin, 1992.

\bibitem{Baxendale1994}
P.~H. Baxendale, A stochastic {H}opf bifurcation, Probab. Theory Relat. Fields
  99 (1994) 581--616.

\bibitem{Gyongy2002}
I.~Gy\"{o}ngy, Approximations of stochastic partial differential equations, in:
  G.~Da~Prato, L.~Tubaro (Eds.), Stochastic Partial Differential Equations,
  Vol. 227 of Lecture Notes in Pure and Appl. Math., Deker, New York, 2002, pp.
  287--307.

\bibitem{Jentzen2011}
A.~Jentzen, P.~Kloeden, Taylor approximations for stochastic partial
  differential equations, SIAM, Philadelphia, 2011.

\bibitem{Biscay1996}
R.~Biscay, J.~C. Jimenez, J.~J. Riera, P.~A. Valdes, Local linearization method
  for the numerical solution of stochastic differential equations, Ann. Inst.
  Statist. Math. 48 (1996) 631��--644.

\bibitem{DeLaCruz2010}
H.~De~la Cruz~Cancino, R.~J. Biscay, J.~C. Jimenez, Carbonell, T.~F.;~Ozaki,
  High order local linearization methods: an approach for constructing a-stable
  explicit schemes for stochastic differential equations with additive noise,
  BIT 50 (2010) 509--539.

\bibitem{Graham2013}
C.~Graham, D.~Talay, Stochastic simulation and Monte Carlo methods.
  Mathematical foundations of stochastic simulation., Springer, Berlin, 2013.

\bibitem{Milstein2004}
G.~N. Milstein, M.~V. Tretyakov, Stochastic numerics for mathematical physics,
  Springer, Berlin, 2004.

\bibitem{Cohen1984}
J.~E. Cohen, C.~M. Newman, The stability of large random matrices and their
  products, Ann. Probab. 12 (1984) 283--310.

\end{thebibliography}

%
%

\end{document}